\documentclass{amsproc}%
\usepackage{amsfonts}%
\usepackage{amsmath}%
\usepackage{amssymb}%
\usepackage{graphicx}
\providecommand{\U}[1]{\protect\rule{.1in}{.1in}}
\theoremstyle{plain}

\newtheorem{definition}{Definition}
\newtheorem{example}{Example}

\newtheorem{lemma}{Lemma}

\newtheorem{remark}{Remark}

\newtheorem{theorem}{Theorem}
\numberwithin{equation}{section}
\begin{document}
\title[A New Type Nonself Total Asymptotically Nonexpansive Mappings]{Fixed Point of A New Type Nonself Total Asymptotically Nonexpansive Mappings
in Banach Spaces}
\author{Birol GUNDUZ}
\curraddr{Department of Mathematics, Faculty of Science and Art, Erzincan University,
Erzincan, 24000, Turkey.}
\email{birolgndz@gmail.com}
\author{Hemen DUTTA}
\curraddr{Department of Mathematics, Gauhati University, Guwahati, India.}
\email{hemen\_dutta08@rediffmail.com}
\author{Adem KILICMAN}
\curraddr{Department of Mathematics and Institute for Mathematical Research,
University Putra Malaysia, Selangor, Malaysia.}
\email{akilic@upm.edu.my}

\begin{abstract}
In this work, a new concept of nonself total asymptotically nonexpansive
mapping is introduced and an iterative process is considered for two nonself
totally asymptotically nonexpansive mappings. Weak and strong convergence
theorems for computing common fixed points of two nonself total asymptotically
nonexpansive mappings are established in the framework of Banach spaces.
Finally, we give an example which show that our theorems are applicable.

\end{abstract}
\maketitle

\section{Introduction}

Let $K$ be a nonempty subset of a real normed linear space $E$. Denote by
$F(T)$ the set of fixed points of $T$, that is, $F(T)=\{x\in K:Tx=x\}$. A
mapping $T:K\rightarrow K$ is called

\begin{itemize}
\item nonexpansive if $T$ satisfy $\left\Vert Tx-Ty\right\Vert \leq\left\Vert
x-y\right\Vert $,

\item asymptotically nonexpansive if there exists a sequence $\left\{
k_{n}\right\}  \in\lbrack1,\infty)$ satisfying $\lim_{n\rightarrow\infty}%
k_{n}=1$ as $n\rightarrow\infty$ such that
\begin{equation}
\left\Vert T^{n}x-T^{n}y\right\Vert \leq k_{n}\left\Vert x-y\right\Vert
,\text{ }n\geq1,\label{d1}%
\end{equation}

\item uniformly $L$-Lipschitzian if there \ exists constant $L\geq0$ such that
$\left\Vert T^{n}x-T^{n}y\right\Vert \leq L\left\Vert x-y\right\Vert $

\item asymptotically quasi-nonexpansive if $F\left(  T\right)  \neq\emptyset$
and there exists a sequence $\{k_{n}\}\subset\lbrack1,\infty)$ with
$\lim_{n\rightarrow\infty}k_{n}=1$ such that
\[
\Vert T^{n}x-p\Vert\leq k_{n}\Vert x-p\Vert
\]
for all $x,y\in K$ and $p\in F\left(  T\right)  $.
\end{itemize}

Note that an asymptotically nonexpansive mapping with a nonempty fixed point
set is an asymptotically quasi-nonexpansive, but the converse of this
statement is not true. As a generalization of nonexpansive mappings, Goebel
and Kirk \cite{goebel} introduced asymptotically nonexpansive mappings. They
also showed that an asymptotically nonexpansive mapping defined on a nonempty
closed bounded subset of a real uniformly convex Banach space has a fixed point.

Iterative techniques for asymptotically nonexpansive self-mapping in Banach
spaces including modified Mann and modified Ishikawa iterations processes have
been studied widely by varied authors. However, these iteration processes may
fail to be well defined when the domain of $T$ is a proper subset of $E$.

If there exists a continuous mapping $P:E\rightarrow K$ such that $Px=x$, for
all $x\in K$, then a subset $K$ of $E$ is said to be a retract of $E$. As an
example of a retract of any uniformly convex Banach space $E$, we can give a
closed convex subset of $E$. The mapping $P$ is called a retraction if it
satisfy $P^{2}=P$. It follows that, if a mapping $P$ is a retraction, then
$Py=y$ for all $y$ in the range of $P$. Intercalarily, if $K$ is closed convex
and $P$ satisfy $P(Px+t(x-Px))=Px$ for all $x\in E$ and for all $t\geq0$, then
$P$ is said to be sunny \cite{bruck}.

In 2003, Chidume et al. \textit{\cite{COZ}} defined a new concept of
asymptotically nonexpansive self-mapping, which is defined as follows:

\begin{definition}
\textit{\cite{COZ} }Let $K$ be a nonempty subset of a real normed space $E$
and $P:E\rightarrow K$ be a nonexpansive retraction of $E$ onto $K$. A nonself
mapping $T:K\rightarrow E$\textit{\ is called asymptotically nonexpansive if
there exists a sequence }$\{k_{n}\}\subset\lbrack1,\infty) $\textit{\ with
}$\lim_{n\rightarrow\infty}k_{n}=1$\textit{\ such that }%
\begin{equation}
\Vert T(PT)^{n-1}x-T(PT)^{n-1}y\Vert\leq k_{n}\Vert x-y\Vert\label{d2}%
\end{equation}
\textit{for all }$x,y\in K$\textit{\ and }$n\geq1$\textit{. }$T$\textit{\ is
called uniformly }$L$\textit{-Lipschitzian if there exists a constant }%
$L>0$\textit{\ such that }%
\[
\Vert T(PT)^{n-1}x-T(PT)^{n-1}y\Vert\leq L\Vert x-y\Vert
\]
\textit{for all }$x,y\in K$\textit{\ and }$n\geq1$.
\end{definition}

Note that if $T$ is a self-mapping, then $P$ is an identity mapping.

To define a new and more general class of nonlinear mappings, Alber et al.
\cite{alber} introduced total asymptotically nonexpansive mappings.

\begin{definition}
\cite{alber} Let $K$ be a nonempty closed subset of a real normed linear space
$X$. A mapping $T:K\rightarrow K$ is called total asymptotically nonexpansive
if there exist nonnegative real sequences $\left\{  \mu_{n}\right\}  ,$
$\left\{  \lambda_{n}\right\}  $ with $\mu_{n},\lambda_{n}\rightarrow0$ as
$n\rightarrow\infty$\ and a strictly increasing continuous function $\phi:%
\mathbb{R}
\rightarrow%
\mathbb{R}
$ with $\phi\left(  0\right)  =0$ such that%
\begin{equation}
\left\Vert T^{n}x-T^{n}y\right\Vert \leq\left\Vert x-y\right\Vert +\mu_{n}%
\phi\left(  \left\Vert x-y\right\Vert \right)  +\lambda_{n},\label{d3}%
\end{equation}
for all $x,y\in K,$ $n\geq1.$
\end{definition}

\begin{remark}
If $\phi\left(  \lambda\right)  =\lambda$ and$\ \lambda_{n}=0$ for all
$n\geq1$, then total asymptotically nonexpansive mappings coincide with
asymptotically nonexpansive mappings. Thus an asymptotically nonexpansive
mappings is a total asymptotically nonexpansive mappings.
\end{remark}

The strongly convergence of iterative processes for a finite family of total
asymptotically nonexpansive mappings in Banach spaces have been studied by
Chidume and Ofoedu \cite{chidume1, chidume2} and they defined nonself total
asymptotically nonexpansive mappings as follow:

\begin{definition}
\cite{chidume1} Let $K$ be a nonempty subset of $E$. Let $P:E\rightarrow K$ be
the nonexpansive retraction of $E$ onto $K.$ A nonself mapping $T:K\rightarrow
E$ is called total asymptotically nonexpansive if there exist nonnegative real
sequences $\left\{  \mu_{n}\right\}  ,$ $\left\{  \lambda_{n}\right\}  $ with
$\mu_{n},\lambda_{n}\rightarrow0$ as $n\rightarrow\infty$\ and a strictly
increasing continuous function $\phi:%
\mathbb{R}
\rightarrow%
\mathbb{R}
$ with $\phi\left(  0\right)  =0$ such that%
\begin{equation}
\left\Vert T(PT)^{n-1}x-T(PT)^{n-1}y\right\Vert \leq\left\Vert x-y\right\Vert
+\mu_{n}\phi\left(  \left\Vert x-y\right\Vert \right)  +\lambda_{n},\text{for
all }x,y\in K,\text{ }n\geq1.\label{nonself total}%
\end{equation}

\end{definition}

\begin{remark}
If $\phi\left(  \lambda\right)  =\lambda$ and$\ \lambda_{n}=0$ for all
$n\geq1$, then nonself total asymptotically nonexpansive mappings coincide
with nonself asymptotically nonexpansive mappings.
\end{remark}

\begin{remark}
$(see$\cite{ZCK}$)$ In case $T:K\rightarrow E$ is asymptotically nonexpansive
and $P:E\rightarrow K$ is a nonexpansive retraction, then $PT:K\rightarrow K$
is asymptotically nonexpansive. In fact we have
\begin{align}
\left\Vert (PT)^{n}x-(PT)^{n}y\right\Vert  & =\left\Vert PT(PT)^{n-1}%
x-PT(PT)^{n-1}y\right\Vert \nonumber\\
& \leq\left\Vert T(PT)^{n-1}x-T(PT)^{n-1}y\right\Vert \label{eq1}\\
& \leq k_{n}\left\Vert x-y\right\Vert ,\nonumber
\end{align}
for all x,$y\in K$ and $n\in%
\mathbb{N}
.$
\end{remark}

But the converseof this is not true. Therefore, Zhou et al. \cite{ZCK}
introduced the following generalized definition recently.

\begin{definition}
\cite{ZCK} Let $K$ be a nonempty subset of real normed linear space $E.$ Let
$P:E\rightarrow K$ be the nonexpansive retraction of $E$ into $K.$ A nonself
mapping $T:K\rightarrow E$ is called asymptotically nonexpansive with respect
to $P$ if there exists sequences $\left\{  k_{n}\right\}  \in\lbrack1,\infty)$
with $k_{n}\rightarrow1$ as $n\rightarrow\infty$ such that%
\begin{equation}
\left\Vert (PT)^{n}x-(PT)^{n}y\right\Vert \leq k_{n}\left\Vert x-y\right\Vert
,\text{ \ }\forall x,y\in K,\text{ }n\in%
\mathbb{N}
.\label{d4}%
\end{equation}
$T$ is said to be uniformly $L$-Lipschitzian with respect to $P$ if there
exists a constant $L\geq0$ such that
\begin{equation}
\left\Vert (PT)^{n}x-(PT)^{n}y\right\Vert \leq L\left\Vert x-y\right\Vert
,\text{ \ }\forall x,y\in K,\text{ }n\in%
\mathbb{N}
.\label{d5*}%
\end{equation}

\end{definition}

Incorporating the above definitions, we introduce the following more general
definition of nonself total asymptotically nonexpansive mapings.

\begin{definition}
Let $K$ be a nonempty subset of $E$. Let $P:E\rightarrow K$ be the
nonexpansive retraction of $E$ onto $K.$ A nonself mapping $T:K\rightarrow E$
is called total asymptotically nonexpansive if there exist nonnegative real
sequences $\left\{  \mu_{n}\right\}  ,$ $\left\{  \lambda_{n}\right\}  $ with
$\mu_{n},\lambda_{n}\rightarrow0$ as $n\rightarrow\infty$\ and a strictly
increasing continuous function $\phi:%
\mathbb{R}
\rightarrow%
\mathbb{R}
$ with $\phi\left(  0\right)  =0$ such that%
\begin{equation}
\left\Vert (PT)^{n}x-(PT)^{n}y\right\Vert \leq\left\Vert x-y\right\Vert
+\mu_{n}\phi\left(  \left\Vert x-y\right\Vert \right)  +\lambda_{n},\text{for
all }x,y\in K,\text{ }n\geq1.\label{d6}%
\end{equation}

\end{definition}

\begin{remark}
We note that if $T:K\rightarrow E$ is a total asymptotically nonexpansive
mapping defined in $(\ref{nonself total})$ and $P:E\rightarrow K$ is a
nonexpansive retraction, then $PT:K\rightarrow K$ is total asymptotically
nonexpansive defined $(\ref{d3})$. Actually using $(\ref{d6})$, we have%
\begin{align*}
\left\Vert (PT)^{n}x-(PT)^{n}y\right\Vert  & \leq\left\Vert T(PT)^{n-1}%
x-T(PT)^{n-1}y\right\Vert \\
& \leq\left\Vert x-y\right\Vert +\mu_{n}\phi\left(  \left\Vert x-y\right\Vert
\right)  +\lambda_{n},\text{for all }x,y\in K,\text{ }n\geq1.
\end{align*}
Conversely, it may not be true.
\end{remark}

\begin{remark}
If $\phi\left(  \lambda\right)  =\lambda$ and$\ \lambda_{n}=0$ for all
$n\geq1$, then $(\ref{d6})$ reduces to $(\ref{d3})$.
\end{remark}

In this paper, we propose an iteration process for calculating common fixed
points of two nonself total asymptotically nonexpansive mappings; and give
convergence criteria for this iteration process for the mappings in Banach
spaces. Some convergence criteria of the iteration process under some
\ restrictions is also established in uniformly convex Banach spaces. We use
the following iteration:

Let $K$ be a nonempty closed convex subset of a real normed linear space $E$
with retraction $P$. Let $T_{1},T_{2}:K\rightarrow E$ be two nonself
asymptotically nonexpansive mappings with respect to $P$.%
\begin{equation}%
\begin{cases}
x_{1}\in K\\
x_{n+1}=(1-\alpha_{n})\left(  PT_{1}\right)  ^{n}y_{n}+\alpha_{n}\left(
PT_{2}\right)  ^{n}y_{n},\\
y_{n}=(1-\beta_{n})x_{n}+\beta_{n}\left(  PT_{1}\right)  ^{n}x_{n},\;\,\,n\in%
\mathbb{N}%
\end{cases}
\label{B}%
\end{equation}
where $\left\{  \alpha_{n}\right\}  $ and $\left\{  \beta_{n}\right\}  $ are
sequences in\textit{\ }$[0,1].$

\section{Preliminaries}

Now we recall some results and concepts.

Let $E$ be a Banach space with its dimension equal or greater than $2$. The
function $\delta_{E}(\varepsilon):(0,2]\rightarrow\lbrack0,1]$ defined by%
\[
\delta_{E}(\varepsilon)=\inf\left\{  1-\left\Vert \frac{1}{2}(x+y)\right\Vert
:\left\Vert x\right\Vert =1,\text{ }\left\Vert y\right\Vert =1,\text{
}\varepsilon=\left\Vert x-y\right\Vert \right\}
\]
is called modulus of $E$. If $\delta_{E}(\varepsilon)>0$ for all
$\varepsilon\in(0,2]$ then $E$ is called uniformly convex.

Let $E$ be a Banach space and $S(E)=\left\{  x\in E:\left\Vert x\right\Vert
=1\right\}  .$ If the limit $\lim_{t\rightarrow0}\frac{\left\Vert
x+ty\right\Vert -\left\Vert x\right\Vert }{t}$ exists for all $x,y\in S(E)$,
then $E$ said to be smooth.

Let $E$ be a Banach space. $E$ is said to have the Opial property if, whenever
$\{x_{n}\}$ is a sequence in $E$ converging weakly to some $x_{0}\in E$ and
$x\neq x_{0}$, it follows that%
\[
\limsup_{n\rightarrow\infty}\left\Vert x_{n}-x\right\Vert <\limsup
_{n\rightarrow\infty}\left\Vert x_{n}-y\right\Vert .
\]

Let $K$ be a nonempty subset of a Banach space $E$. For $x\in K$, the inward
set of $x$ is defined by $\{x+\lambda(u-x):u\in K,\lambda\geq1\}$ and

it is indicated by $I_{K}(x)$. Let $cl[I_{K}(x)]$ denotes the closure of the
inward set. A mapping $T:K\rightarrow E$ is called weakly inward if $Tx\in
cl[I_{K}(x)]$ for all $x\in K$. We can give self-mappings as example of weakly
inward mappings.

We denote domain and range of a mapping $T$ in $E$ by $D(T)$ and $R(T)$,
respectively. $T$ is said to be demiclosed at $p$ if for any given sequence
$\{x_{n}\}$ in $D(T)$, the conditions $x_{n}\rightharpoonup x_{0}$ and
$Tx_{n}\rightarrow p$ imply $Tx_{0}=p$, where $x_{n}\rightharpoonup x_{0}$
means that $\{x_{n}\}$ converges weakly to $x_{0}.$

Let $T:K\rightarrow K$~be a mapping. $T:K\rightarrow K$ is called

(i) completely continuous if for every bounded sequence $\{x_{n}\}$, there
exists a subsequence $\{x_{n_{j}}\}$ of $\{x_{n}\}$ such that $\{Tx_{n_{j}}\}
$ converges to some element in $R(T)$,

(ii) demi-compact if any sequence $\{x_{n}\}$ in $K$ satisfying $x_{n}%
-Tx_{n}\rightarrow0$ as $n\rightarrow\infty$ has a convergent subsequence.

(iii) satisfy condition $(A)$ \cite{ST} if $F\left(  T\right)  \neq\emptyset$
and there is a nondecreasing function $f:\left[  0,\infty\right)
\rightarrow\left[  0,\infty\right)  $ with $f\left(  0\right)  =0$, $f\left(
t\right)  >0$ for all $t\in\left(  0,\infty\right)  $ such that $\left\Vert
x-Tx\right\Vert \geq f\left(  d\left(  x,F\left(  T\right)  \right)  \right)
$ for all $x\in K$, where $d\left(  x,F(T)\right)  =\inf\left\{  \left\Vert
x-p\right\Vert :p\in F(T)\right\}  $.

We use modified version of condition (A) defined by Khan and
Fukharuddin\ \cite{KU} for two mappings as follows:

Two mappings $T_{1},$ $T_{2}:K\rightarrow K$ are said to satisfy condition
$(A^{\prime})$ \cite{KU} if there is a nondecreasing function $f:\left[
0,\infty\right)  \rightarrow\left[  0,\infty\right)  $ with $f\left(
0\right)  =0$, $f\left(  t\right)  >0$ for all $t\in\left(  0,\infty\right)  $
such that%
\[
\frac{1}{2}\left(  \left\Vert x-PT_{1}x\right\Vert +\left\Vert x-PT_{2}%
x\right\Vert \right)  \geq f\left(  d\left(  x,F\right)  \right)
\]
for all $x\in K$.

It is point out that condition $(A^{\prime})$\ reduces to condition $(A)$ when
$T_{1}$ equal to $T_{2}$. Also condition $(A)$ is weaker than demicompactness
or semicompactness of $T$, see \cite{ST}.

For proving main theorems of our paper, we need folowing lemmas.

\begin{lemma}
\label{exist}\cite{tan} If $\{a_{n}\},\{b_{n}\}$ and $\{c_{n}\}$ are\ three
sequences of nonnegative real numbers such that%
\[
a_{n+1}\leq\left(  1+b_{n}\right)  a_{n}+c_{n}\mathrm{,\;\ }n\in%
\mathbb{N}%
\]
and the sum $\sum_{n=1}^{\infty}b_{n}$ and $\sum_{n=1}^{\infty}c_{n}$ are
finite, then $\lim\limits_{n\rightarrow\infty}a_{n}$ exist.
\end{lemma}

\begin{lemma}
\label{fark}\cite{schu}\textit{Suppose that }$E$\textit{\ is a uniformly
convex Banach space and }$0<p\leq t_{n}\leq q<1$\textit{\ for all }$n~\in%
\mathbb{N}
$\textit{. Also, suppose that }$\left\{  x_{n}\right\}  $\textit{\ and
}$\left\{  y_{n}\right\}  $\textit{\ are sequences of }$E$ \textit{such that}%
\[
\limsup_{n\rightarrow\infty}\left\Vert x_{n}\right\Vert \leq r,\text{ }%
\limsup_{n\rightarrow\infty}\left\Vert y_{n}\right\Vert \leq r\text{ and }%
\lim_{n\rightarrow\infty}\left\Vert \left(  1-t_{n}\right)  x_{n}+t_{n}%
y_{n}\right\Vert =r
\]
\textit{hold for some }$r\geq0$\textit{. Then }$\lim_{n\rightarrow\infty
}\left\Vert x_{n}-y_{n}\right\Vert =0$\textit{.}
\end{lemma}

\begin{lemma}
\label{F}\cite{TA}Let $E$ be real smooth Banach space, let $K$ be nonempty
closed convex subset of $E$ with $P$ as a sunny nonexpansive retraction, and
let $T:K\rightarrow E$ be a mapping satisfying weakly inward condition. Then
$F(PT)=F(T).$
\end{lemma}

\section{Main results}

In this section, we shall prove convergence of the iteration sceheme defined
(\ref{B}) for two nonself total asymptotically nonexpansive mappings. We
always assume $F=F(T_{1})\cap F(T_{2})=\left\{  x\in K:T_{1}x=T_{2}%
x=x\right\}  \neq\emptyset$. In order to prove our main results, we shall make
use of following lemmas.

\begin{lemma}
Let $E$ be a real Banach space, $K$\ be a nonempty subset of $E$ which as also
a nonexpansive retract with retraction $P$. Let $T_{1},T_{2}:K\rightarrow E$
be two nonself total asymptotically nonexpansive mappings with squences
$\{\mu_{n}\},\{\lambda_{n}\}$. Suppose that there exist $M,M^{\ast}>0,$ such
that $\phi(\kappa)\leq M^{\ast}\kappa$ for all $\kappa\geq M$. Then for any
$x,y\in K$ following inequalties hold;%
\begin{align}
\left\Vert \left(  PT_{1}\right)  ^{n}x-\left(  PT_{1}\right)  ^{n}%
y\right\Vert  & \leq(1+\mu_{n}M^{\ast})\left\Vert x-y\right\Vert +\mu_{n}%
\phi\left(  M\right)  +\lambda_{n},\label{1a}\\
\left\Vert \left(  PT_{2}\right)  ^{n}x-\left(  PT_{2}\right)  ^{n}%
y\right\Vert  & \leq(1+\mu_{n}M^{\ast})\left\Vert x-y\right\Vert +\mu_{n}%
\phi\left(  M\right)  +\lambda_{n},\text{ }n\geq1.\label{1b}%
\end{align}

\end{lemma}

\begin{proof}
Since $T_{1},T_{2}:K\rightarrow E$ are two nonself total asymptotically
nonexpansive mappings with squences $\{\mu_{n}\},\{\lambda_{n}\}$, then we
have for all $x,y\in K,$%
\begin{equation}
\left\Vert \left(  PT_{1}\right)  ^{n}x-\left(  PT_{1}\right)  ^{n}%
y\right\Vert \leq\left\Vert x-y\right\Vert +\mu_{n}\phi\left(  \left\Vert
x-y\right\Vert \right)  +\lambda_{n},\label{2a}%
\end{equation}%
\begin{equation}
\left\Vert \left(  PT_{2}\right)  ^{n}x-\left(  PT_{2}\right)  ^{n}%
y\right\Vert \leq\left\Vert x-y\right\Vert +\mu_{n}\phi\left(  \left\Vert
x-y\right\Vert \right)  +\lambda_{n},n\geq1.\label{2b}%
\end{equation}
Since $\phi:{{\mathbb{R}}}^{+}\rightarrow{{\mathbb{R}}}^{+}$ is a strictly
increasing continuous function with $\phi\left(  0\right)  =0,$ we get
$\phi\left(  \kappa\right)  \leq\phi\left(  \kappa\right)  $ whenever
$\kappa\leq M$ and $\phi\left(  \kappa\right)  \leq M^{\ast}\kappa$ if
$\kappa\geq M.$ In two case, we get%
\begin{equation}
\phi\left(  \kappa\right)  \leq\phi\left(  M_{1}\right)  +M_{1}^{\ast}%
\kappa,\text{ \ }\phi\left(  \sigma\right)  \leq\phi\left(  M_{2}\right)
+M_{2}^{\ast}\sigma\label{2}%
\end{equation}
for some $M,M^{\ast}\geq0.$ The (\ref{2a}) and (\ref{2b}) with (\ref{2}) yield
that%
\begin{align*}
\left\Vert \left(  PT_{1}\right)  ^{n}x-\left(  PT_{1}\right)  ^{n}%
y\right\Vert  & \leq(1+\mu_{n}M^{\ast})\left\Vert x-y\right\Vert +\mu_{n}%
\phi\left(  M\right)  +\lambda_{n},\\
\left\Vert \left(  PT_{2}\right)  ^{n}x-\left(  PT_{2}\right)  ^{n}%
y\right\Vert  & \leq(1+\mu_{n}M^{\ast})\left\Vert x-y\right\Vert +\mu_{n}%
\phi\left(  M\right)  +\lambda_{n},\text{ }n\geq1.
\end{align*}
This completes the proof.
\end{proof}

\begin{lemma}
\label{Lemma1}Let $E$ be a real Banach space, $K$\ be a closed convex
nonempty\ subset of $E$ which as also a nonexpansive retract with retraction
$P$. Let $T_{1},T_{2}:K\rightarrow E$ be two nonself total asymptotically
nonexpansive mappings with squences $\{\mu_{n}\},\{\lambda_{n}\}$ such that $%
{\textstyle\sum\nolimits_{n=1}^{\infty}}
\mu_{n}<\infty,%
{\textstyle\sum\nolimits_{n=1}^{\infty}}
\lambda_{n}<\infty$. Suppose that there exist $M,M^{\ast}>0,~$such that
$\phi(\kappa)\leq M^{\ast}\kappa$ for all $\kappa\geq M$. Suppose that
$\{\alpha_{n}\}$ and $\left\{  \beta_{n}\right\}  $\ are two real sequences in
$[0,1],$ $F\neq\emptyset$ and $\{x_{n}\}$ is defined by $(\ref{B})$. Then the
limits $\lim_{n\rightarrow\infty}\left\Vert x_{n}-p\right\Vert $ and
$\lim_{n\rightarrow\infty}d\left(  x_{n},F\right)  $ exists, where $d\left(
x_{n},F\right)  =\inf_{p\in F}\left\Vert x_{n}-p\right\Vert .$
\end{lemma}

\begin{proof}
Let $p\in F.$ It follows from (\ref{B}) and (\ref{1a}) that%
\begin{align}
\left\Vert y_{n}-p\right\Vert  & =\left\Vert (1-\beta_{n})x_{n}+\beta
_{n}\left(  PT_{1}\right)  ^{n}x_{n}-p\right\Vert \nonumber\\
& \leq(1-\beta_{n})\left\Vert x_{n}-p\right\Vert +\beta_{n}\left\Vert \left(
PT_{1}\right)  ^{n}x_{n}-p\right\Vert \nonumber\\
& \leq(1-\beta_{n})\left\Vert x_{n}-p\right\Vert +\beta_{n}\left[  (1+\mu
_{n}M^{\ast})\left\Vert x-p\right\Vert +\mu_{n}\phi\left(  M\right)
+\lambda_{n}\right] \nonumber\\
& \leq\left\Vert x_{n}-p\right\Vert +\beta_{n}\mu_{n}M^{\ast}\left\Vert
x_{n}-p\right\Vert +\mu_{n}\phi\left(  M\right)  +{\lambda}_{n}\nonumber\\
& \leq\left(  1+\mu_{n}M^{\ast}\right)  \left\Vert x_{n}-p\right\Vert
+\phi\left(  M\right)  \mu_{n}+{\lambda}_{n}\label{3}%
\end{align}
Similarly, from (\ref{B}), (\ref{1a}) and (\ref{1b}) we have%
\begin{align}
\left\Vert x_{n+1}-p\right\Vert  & =\left\Vert (1-\alpha_{n})\left(
PT_{1}\right)  ^{n}y_{n}+\alpha_{n}\left(  PT_{2}\right)  ^{n}y_{n}%
-p\right\Vert \nonumber\\
& \leq(1-\alpha_{n})\left\Vert \left(  PT_{1}\right)  ^{n}y_{n}-p\right\Vert
+\alpha_{n}\left\Vert \left(  PT_{2}\right)  ^{n}y_{n}-p\right\Vert
\nonumber\\
& \leq(1-\alpha_{n})\left[  (1+\mu_{n}M^{\ast})\left\Vert y_{n}-p\right\Vert
+\mu_{n}\phi\left(  M\right)  +\lambda_{n}\right] \nonumber\\
& +\alpha_{n}[(1+\mu_{n}M^{\ast})\left\Vert y_{n}-p\right\Vert +\mu_{n}%
\phi\left(  M\right)  +\lambda_{n}]\nonumber\\
& \leq(1+\mu_{n}M^{\ast})\left\Vert y_{n}-p\right\Vert +\mu_{n}\phi\left(
M\right)  +\lambda_{n}\nonumber\\
& \leq(1+\mu_{n}M^{\ast})\left[  (1+\mu_{n}M^{\ast})\left\Vert x_{n}%
-p\right\Vert +\mu_{n}\phi\left(  M\right)  +\lambda_{n}\right] \nonumber\\
& \leq\left(  1+2\mu_{n}M^{\ast}+\left(  \mu_{n}M^{\ast}\right)  ^{2}\right)
\left\Vert x_{n}-p\right\Vert \nonumber\\
& +\phi\left(  M\right)  M^{\ast}\mu_{n}^{2}+M^{\ast}\lambda_{n}\mu_{n}%
M^{\ast}+\mu_{n}\phi\left(  M\right)  +\lambda_{n}\nonumber\\
& \leq\left(  1+b_{n}\right)  \left\Vert x_{n}-p\right\Vert +c_{n}\label{4}%
\end{align}
where $b_{n}=2\mu_{n}M_{1}^{\ast}+\left(  \mu_{n}M_{1}^{\ast}\right)  ^{2}$
and $c_{n}=\phi\left(  M_{1}\right)  M_{1}^{\ast}\mu_{n}^{2}+M_{1}^{\ast
}\lambda_{n}\mu_{n}M_{1}^{\ast}+\mu_{n}\phi\left(  M_{1}\right)  +\lambda_{n}%
$. Since (\ref{4}) is true for each $p$ in $F.$ This implies that
\begin{equation}
d(x_{n+1},p)\leq\left(  1+b_{n}\right)  d(x_{n},p)+c_{n}.\label{5}%
\end{equation}
Since $%
{\textstyle\sum\nolimits_{n=1}^{\infty}}
b_{n}<\infty$ and $%
{\textstyle\sum\nolimits_{n=1}^{\infty}}
c_{n}<\infty,$ from Lemma \ref{exist}, we obtain that $\lim
\limits_{n\rightarrow\infty}\left\Vert x_{n}-p\right\Vert $ and $\lim
_{n\rightarrow\infty}d\left(  x_{n},F\right)  $ exists. This completes the proof.
\end{proof}

Now we have enough knowledge to formulate and prove a criterion on strong
convergence of $\{x_{n}\}$ given by (\ref{B}).

\begin{theorem}
\label{Teorem}Let $E$ be a real Banach space, $K$\ be a nonempty closed convex
subset of $E$ which as also a nonexpansive retract with retraction $P$. Let
$T_{1},T_{2}:K\rightarrow E$ be two continuous nonself total asymptotically
nonexpansive mappings with squences $\{\mu_{n}\},\{\lambda_{n}\}$ such that $%
{\textstyle\sum\nolimits_{n=1}^{\infty}}
\mu_{n}<\infty,%
{\textstyle\sum\nolimits_{n=1}^{\infty}}
\lambda_{n}<\infty$. Suppose that there exist $M,M^{\ast}>0$ such that
$\phi(\kappa)\leq M^{\ast}\kappa$ for all $\kappa\geq M$. Suppose that
$\{\alpha_{n}\}$ and $\left\{  \beta_{n}\right\}  $\ are two real sequences in
$[0,1],$ $F\neq\emptyset$ and $\{x_{n}\}$ is defined by $(\ref{B})$. Then the
sequence $\left\{  x_{n}\right\}  $ strongly converges to a common fixed point
of $T_{1}$ and $T_{2}$ if and only if $\lim\inf_{n\rightarrow\infty}d\left(
x_{n},F\right)  =0.$
\end{theorem}

\begin{proof}
The necessity of the conditions is obvious. Therefore, we give proof for the
sufficiency. Assume that $\lim\inf_{n\rightarrow\infty}d\left(  x_{n}%
,F\right)  =0.$ From\ Lemma \ref{Lemma1}, $\lim_{n\rightarrow\infty}d\left(
x_{n},F\right)  $ exists. Our hypothesis implies $\lim\inf_{n\rightarrow
\infty}d\left(  x_{n},F\right)  =0$, therefore we get $\lim_{n\rightarrow
\infty}d\left(  x_{n},F\right)  =0$.

Now we shall show that $\{x_{n}\}$ is a Cauchy sequence in $E$. As
$1+t\leq\exp\left(  t\right)  $ for all $t>0,$ from (\ref{5}), we obtain%
\begin{equation}
\left\Vert x_{n+1}-p\right\Vert \leq\exp b_{n}\left(  \left\Vert
x_{n}-p\right\Vert +c_{n}\right)  .\label{6}%
\end{equation}
Thus, for any given $m,n$, iterating (\ref{6}), we obtain%
\begin{align*}
\left\Vert x_{n+m}-p\right\Vert  & \leq\exp b_{n+m-1}\left(  \left\Vert
x_{n+m-1}-p\right\Vert +{c}_{n+m-1}\right) \\
& \vdots\\
& \leq\exp(%
{\textstyle\sum\limits_{i=n}^{n+m-1}}
b_{i})\left(  \left\Vert x_{n}-p\right\Vert +%
{\textstyle\sum\limits_{i=n}^{n+m-1}}
c_{i}\right) \\
& \leq\exp(%
{\textstyle\sum\limits_{i=n}^{\infty}}
b_{i})\left(  \left\Vert x_{n}-p\right\Vert +%
{\textstyle\sum\limits_{i=n}^{\infty}}
{c}_{i}\right)  .
\end{align*}
Therefore,%
\begin{align*}
\left\Vert x_{n+m}-x_{n}\right\Vert  & \leq\left\Vert x_{n+m}-p\right\Vert
+\left\Vert x_{n}-p\right\Vert \\
& \leq\left[  1+\left(  \exp(%
{\textstyle\sum\limits_{i=n}^{\infty}}
b_{i})\right)  \right]  \left\Vert x_{n}-p\right\Vert \\
& +\exp(%
{\textstyle\sum\limits_{i=n}^{\infty}}
b_{i})\left(
{\textstyle\sum\limits_{i=n}^{\infty}}
c_{i}\right)  .
\end{align*}
So
\begin{equation}
\left\Vert x_{n+m}-x_{n}\right\Vert \leq D\left\Vert x_{n}-p\right\Vert
+Dc_{i}\label{7}%
\end{equation}
for a real number $D>0.$ Taking infimum on $p\in F$ in (\ref{7}), we get%
\[
\left\Vert x_{n+m}-x_{n}\right\Vert \leq Dd\left(  x_{n},F\right)  +D\left(
{\textstyle\sum\limits_{i=n}^{\infty}}
c_{i}\right)
\]
For given $\epsilon>0,$ using $\lim_{n\rightarrow\infty}d\left(
x_{n},F\right)  =0$ and $%
{\textstyle\sum\limits_{i=n}^{\infty}}
c_{i}<\infty,$ there exists an integer $N_{1}>0$ such that for all $n\geq
N_{1},$ $d\left(  x_{n},F\right)  <\epsilon/2D$ and $%
{\textstyle\sum\limits_{i=n}^{\infty}}
c_{i}<\epsilon/2D.$ Consequently, from last inequality we have
\begin{align*}
\left\Vert x_{n+m}-x_{n}\right\Vert  & \leq Dd\left(  x_{n},F\right)
+D\left(
{\textstyle\sum\limits_{i=n}^{\infty}}
c_{i}\right) \\
& \leq D\frac{\epsilon}{2D}+D\frac{\epsilon}{2D}=\epsilon
\end{align*}
which means that $\left\{  x_{n}\right\}  $ is a Cauchy sequence. Since the
space $E$ is complete, thus $\lim_{n\rightarrow\infty}x_{n}$ exists. Let
$\lim_{n\rightarrow\infty}x_{n}=q.$ Since $T_{1}$ and $T_{2}~$are\ continuous
mappings, the set of common fixed points of $T_{1}$ and $T_{2}$ is closed. We
now show that $q\in F.$ Suppose that $q\notin F.$ Since $F$ is closed subset
of $E,$ we have that $d\left(  q,F\right)  >0.$ But, for all $p\in F,$ we have%
\[
\left\Vert q-p\right\Vert \leq\left\Vert q-x_{n}\right\Vert +\left\Vert
x_{n}-p\right\Vert .
\]
This inequality gives%
\[
d\left(  q,F\right)  \leq\left\Vert q-x_{n}\right\Vert +d\left(
x_{n},F\right)  ,
\]
and so we get $d\left(  q,F\right)  =0$ as $n\rightarrow\infty,$ which
contradicts the fact $d\left(  q,F\right)  >0.$ Hence, $q\in F.$ This
completes the proof of the theorem.
\end{proof}

For our next theorems, we start by proving the following lemma which will be
used in the sequel.

\begin{lemma}
\label{Lemma2}Let $E$ be a uniformly convex real Banach space, $K$\ be a
nonempty closed convex subset of $E$ which as also a nonexpansive retract with
retraction $P$. Let $T_{1},T_{2}:K\rightarrow E$ be two uniformly
$L$-Lipschitzian, nonself total asymptotically nonexpansive mappings with
squences $\{\mu_{n}\},\{\lambda_{n}\}$ such that $%
{\textstyle\sum\nolimits_{n=1}^{\infty}}
\mu_{n}<\infty,%
{\textstyle\sum\nolimits_{n=1}^{\infty}}
\lambda_{n}<\infty$. Suppose that there exist $M,M^{\ast}>0,$ such that
$\phi(\kappa)\leq M^{\ast}\kappa$. Let $\{\alpha_{n}\}$ and $\left\{
\beta_{n}\right\}  $\ be two real sequences in $[\varepsilon,1-\varepsilon],$
for some $\varepsilon\in(0,1).$ Suppose that $\{x_{n}\}$ is genareted
iteratively by $(\ref{B})$. Then
\[
\lim_{n\rightarrow\infty}\left\Vert x_{n}-T_{1}x_{n}\right\Vert =\lim
_{n\rightarrow\infty}\left\Vert x_{n}-T_{2}x_{n}\right\Vert =0.
\]

\end{lemma}

\begin{proof}
For any given $p\in F,$ by Lemma \ref{Lemma1} , we know that $\lim
_{n\rightarrow\infty}\left\Vert x_{n}-p\right\Vert $ exists. Suppose
$\lim_{n\rightarrow\infty}\left\Vert x_{n}-p\right\Vert =c,$ for some
$c\geq0.$ If $c=0,$ there is no anything to prove. Assume $c>0$. Taking
$\lim\sup$ on (\ref{3}), we have%
\begin{equation}
\lim\sup_{n\rightarrow\infty}\left\Vert y_{n}-p\right\Vert \leq\lim
\sup_{n\rightarrow\infty}\left(  1+\mu_{n}M^{\ast}\right)  \left\Vert
x_{n}-p\right\Vert +\phi\left(  M\right)  \mu_{n}+{\lambda}_{n}=c\label{8}%
\end{equation}
Therefore $\left\Vert \left(  PT_{1}\right)  ^{n}y_{n}-p\right\Vert \leq
(1+\mu_{n}M^{\ast})\left\Vert y_{n}-p\right\Vert +\mu_{n}\phi\left(  M\right)
+\lambda_{n}$ for all $n\geq1$ implies that%
\begin{equation}
\lim\sup_{n\rightarrow\infty}\left\Vert \left(  PT_{1}\right)  ^{n}%
y_{n}-p\right\Vert \leq c.\label{9}%
\end{equation}
Similar way, we get%
\begin{equation}
\lim\sup_{n\rightarrow\infty}\left\Vert \left(  PT_{2}\right)  ^{n}%
y_{n}-p\right\Vert \leq\lim\sup_{n\rightarrow\infty}(1+\mu_{n}M^{\ast
})\left\Vert y_{n}-p\right\Vert +\mu_{n}\phi\left(  M\right)  +\lambda_{n}\leq
c\label{10}%
\end{equation}
In addition,
\begin{align}
c  & =\lim_{n\rightarrow\infty}\left\Vert x_{n+1}-p\right\Vert \nonumber\\
& =\lim_{n\rightarrow\infty}\left\Vert (1-\alpha_{n})\left(  PT_{1}\right)
^{n}y_{n}+\alpha_{n}\left(  PT_{2}\right)  ^{n}y_{n}-p\right\Vert \nonumber\\
& =\lim_{n\rightarrow\infty}\left\Vert (1-\alpha_{n})\left(  \left(
PT_{1}\right)  ^{n}y_{n}-p\right)  +\alpha_{n}\left(  \left(  PT_{2}\right)
^{n}y_{n}-p\right)  \right\Vert \label{11}%
\end{align}
which gives that
\begin{equation}
\lim_{n\rightarrow\infty}\left\Vert (1-\alpha_{n})\left(  \left(
PT_{1}\right)  ^{n}y_{n}-p\right)  +\alpha_{n}\left(  \left(  PT_{2}\right)
^{n}y_{n}-p\right)  \right\Vert =c.\label{12}%
\end{equation}
Now using (\ref{9}) with (\ref{10}) and applying Lemma \ref{fark} to
(\ref{12}) , we obtain%
\begin{equation}
\lim_{n\rightarrow\infty}\left\Vert \left(  PT_{1}\right)  ^{n}y_{n}-\left(
PT_{2}\right)  ^{n}y_{n}\right\Vert =0\label{13}%
\end{equation}
Noting that%
\begin{align}
\left\Vert x_{n+1}-p\right\Vert  & =\left\Vert (1-\alpha_{n})\left(
PT_{1}\right)  ^{n}y_{n}+\alpha_{n}\left(  PT_{2}\right)  ^{n}y_{n}%
-p\right\Vert \nonumber\\
& \leq\left\Vert \left(  PT_{1}\right)  ^{n}y_{n}-p\right\Vert +\alpha
_{n}\left\Vert \left(  PT_{2}\right)  ^{n}y_{n}-\left(  PT_{1}\right)
^{n}y_{n}\right\Vert \nonumber\\
& \leq(1+\mu_{n}M^{\ast})\left\Vert y_{n}-p\right\Vert +\mu_{n}\phi\left(
M\right)  +\lambda_{n}\nonumber\\
& +\alpha_{n}\left\Vert \left(  PT_{2}\right)  ^{n}y_{n}-\left(
PT_{1}\right)  ^{n}y_{n}\right\Vert ,\label{14}%
\end{align}
this implies that%
\begin{equation}
c\leq\lim\inf_{n\rightarrow\infty}\left\Vert y_{n}-p\right\Vert \label{15}%
\end{equation}
From (\ref{8}) and (\ref{15}), we get%
\begin{equation}
\lim_{n\rightarrow\infty}\left\Vert y_{n}-p\right\Vert =c\label{15a}%
\end{equation}
Otherwise, $\left\Vert \left(  PT_{1}\right)  ^{n}x_{n}-p\right\Vert
\leq(1+\mu_{n}M^{\ast})\left\Vert x_{n}-p\right\Vert +\mu_{n}\phi\left(
M\right)  +\lambda_{n}$ for all $n\geq1$ implies that%
\begin{equation}
\lim\sup_{n\rightarrow\infty}\left\Vert \left(  PT_{1}\right)  ^{n}%
x_{n}-p\right\Vert \leq c.\label{16}%
\end{equation}
Consequently,%
\begin{align}
c  & =\lim_{n\rightarrow\infty}\left\Vert y_{n}-p\right\Vert \nonumber\\
& =\lim_{n\rightarrow\infty}\left\Vert (1-\beta_{n})x_{n}+\beta_{n}\left(
PT_{1}\right)  ^{n}x_{n}-p\right\Vert \nonumber\\
& =\lim_{n\rightarrow\infty}\left\Vert (1-\beta_{n})\left(  x_{n}-p\right)
+\beta_{n}\left(  \left(  PT_{1}\right)  ^{n}x_{n}-p\right)  \right\Vert
\label{17}%
\end{align}
gives that%
\begin{equation}
\lim_{n\rightarrow\infty}\left\Vert (1-\beta_{n})\left(  x_{n}-p\right)
+\beta_{n}\left(  \left(  PT_{1}\right)  ^{n}x_{n}-p\right)  \right\Vert
=c.\label{18}%
\end{equation}
Again using Lemma \ref{fark}, we get%
\begin{equation}
\lim_{n\rightarrow\infty}\left\Vert \left(  PT_{1}\right)  ^{n}x_{n}%
-x_{n}\right\Vert =0\label{19}%
\end{equation}
Moreover, from (\ref{B}), we have%
\begin{align}
\left\Vert y_{n}-x_{n}\right\Vert  & =\left\Vert (1-\beta_{n})x_{n}+\beta
_{n}\left(  PT_{1}\right)  ^{n}x_{n}-x_{n}\right\Vert \nonumber\\
& =\beta_{n}\left\Vert \left(  PT_{1}\right)  ^{n}x_{n}-x_{n}\right\Vert
\label{20}%
\end{align}
Thus from (\ref{19})
\begin{equation}
\lim_{n\rightarrow\infty}\left\Vert y_{n}-x_{n}\right\Vert =0.\label{21}%
\end{equation}
Also
\begin{align}
\left\Vert \left(  PT_{2}\right)  ^{n}y_{n}-x_{n}\right\Vert  & \leq\left\Vert
\left(  PT_{2}\right)  ^{n}y_{n}-\left(  PT_{1}\right)  ^{n}y_{n}\right\Vert
+\left\Vert \left(  PT_{1}\right)  ^{n}y_{n}-\left(  PT_{1}\right)  ^{n}%
x_{n}\right\Vert \nonumber\\
& +\left\Vert \left(  PT_{1}\right)  ^{n}x_{n}-x_{n}\right\Vert \nonumber\\
& \leq\left\Vert \left(  PT_{2}\right)  ^{n}y_{n}-\left(  PT_{1}\right)
^{n}y_{n}\right\Vert +(1+\mu_{n}M^{\ast})\left\Vert x_{n}-y_{n}\right\Vert
\nonumber\\
& +\mu_{n}\phi\left(  M\right)  +\lambda_{n}+\left\Vert \left(  PT_{1}\right)
^{n}x_{n}-x_{n}\right\Vert \label{22}%
\end{align}
yields from (\ref{13}), (\ref{19}) and (\ref{21}) that%
\begin{equation}
\lim_{n\rightarrow\infty}\left\Vert \left(  PT_{2}\right)  ^{n}y_{n}%
-x_{n}\right\Vert =0\label{23}%
\end{equation}
By (\ref{21}) and (\ref{23}), we get%
\begin{align}
\left\Vert \left(  PT_{2}\right)  ^{n}x_{n}-x_{n}\right\Vert  & \leq\left\Vert
\left(  PT_{2}\right)  ^{n}x_{n}-\left(  PT_{2}\right)  ^{n}y_{n}\right\Vert
+\left\Vert \left(  PT_{2}\right)  ^{n}y_{n}-x_{n}\right\Vert \nonumber\\
& \leq(1+\mu_{n}M^{\ast})\left\Vert x_{n}-y_{n}\right\Vert +\mu_{n}\phi\left(
M\right)  +\lambda_{n}\nonumber\\
& +\left\Vert \left(  PT_{2}\right)  ^{n}y_{n}-x_{n}\right\Vert \label{24}%
\end{align}
and this implies%
\begin{equation}
\lim_{n\rightarrow\infty}\left\Vert \left(  PT_{2}\right)  ^{n}x_{n}%
-x_{n}\right\Vert =0.\label{25}%
\end{equation}
Then
\begin{align}
\left\Vert \left(  PT_{1}\right)  ^{n}y_{n}-x_{n}\right\Vert  & \leq\left\Vert
\left(  PT_{1}\right)  ^{n}y_{n}-\left(  PT_{1}\right)  ^{n}x_{n}\right\Vert
+\left\Vert \left(  PT_{1}\right)  ^{n}x_{n}-x_{n}\right\Vert \nonumber\\
& \leq(1+\mu_{n}M^{\ast})\left\Vert x_{n}-y_{n}\right\Vert +\mu_{n}\phi\left(
M\right)  +\lambda_{n}\nonumber\\
& +\left\Vert \left(  PT_{1}\right)  ^{n}x_{n}-x_{n}\right\Vert \label{26}%
\end{align}
gives%
\begin{equation}
\lim_{n\rightarrow\infty}\left\Vert \left(  PT_{1}\right)  ^{n}y_{n}%
-x_{n}\right\Vert =0.\label{27}%
\end{equation}
Using (\ref{13}) and (\ref{27}), we have%
\begin{align}
\left\Vert x_{n+1}-x_{n}\right\Vert  & =\left\Vert (1-\alpha_{n})\left(
PT_{1}\right)  ^{n}y_{n}+\alpha_{n}\left(  PT_{2}\right)  ^{n}y_{n}%
-x_{n}\right\Vert \nonumber\\
& \leq\left\Vert \left(  PT_{1}\right)  ^{n}y_{n}-x_{n}\right\Vert +\alpha
_{n}\left\Vert \left(  PT_{1}\right)  ^{n}y_{n}-\left(  PT_{2}\right)
^{n}y_{n}\right\Vert \label{28}\\
& \rightarrow0\text{ as }n\rightarrow\infty.\nonumber
\end{align}
It follows (\ref{19}) and (\ref{28}) that%
\begin{align}
\left\Vert x_{n}-\left(  PT_{1}\right)  ^{n-1}x_{n}\right\Vert  &
\leq\left\Vert x_{n}-x_{n-1}\right\Vert +\left\Vert x_{n-1}-\left(
PT_{1}\right)  ^{n-1}x_{n-1}\right\Vert \nonumber\\
& +\left\Vert \left(  PT_{1}\right)  ^{n-1}x_{n-1}-\left(  PT_{1}\right)
^{n-1}x_{n}\right\Vert \nonumber\\
& \leq\left\Vert x_{n}-x_{n-1}\right\Vert +\left\Vert x_{n-1}-\left(
PT_{1}\right)  ^{n-1}x_{n-1}\right\Vert \label{29}\\
& +(1+\mu_{n-1}M^{\ast})\left\Vert x_{n}-x_{n-1}\right\Vert +\mu_{n-1}%
\phi\left(  M\right)  +\lambda_{n-1}\nonumber\\
& \rightarrow0\text{ as }n\rightarrow\infty.\nonumber
\end{align}
Since $T_{1}$ is uniformly $L$-Lipschitzian, it follows from (\ref{29}) that%
\begin{equation}
\left\Vert \left(  PT_{1}\right)  ^{n}x_{n}-\left(  PT_{1}\right)
x_{n}\right\Vert \leq L\left\Vert \left(  PT_{1}\right)  ^{n-1}x_{n}%
-x_{n}\right\Vert \rightarrow0\text{ as }n\rightarrow\infty.\label{30}%
\end{equation}
Thus, by (\ref{19}) and (\ref{30}), we get%
\begin{align}
\left\Vert x_{n}-\left(  PT_{1}\right)  x_{n}\right\Vert  & \leq\left\Vert
x_{n}-\left(  PT_{1}\right)  ^{n}x_{n}\right\Vert +\left\Vert x_{n}-\left(
PT_{1}\right)  ^{n}x_{n}\right\Vert \nonumber\\
& +\left\Vert \left(  PT_{1}\right)  ^{n}x_{n}-\left(  PT_{1}\right)
x_{n}\right\Vert \label{31}%
\end{align}
and so%
\begin{equation}
\lim_{n\rightarrow\infty}\left\Vert x_{n}-\left(  PT_{1}\right)
x_{n}\right\Vert =0.\label{32}%
\end{equation}
In the same way, we obtain%
\begin{equation}
\lim_{n\rightarrow\infty}\left\Vert x_{n}-\left(  PT_{2}\right)
x_{n}\right\Vert =0.\label{33}%
\end{equation}
This completes the proof of lemma.
\end{proof}

\begin{theorem}
Let $K$ be a nonempty closed convex subset of a real uniformly convex and
smooth Banach space $E$ with $P$ as a sunny nonexpansive retraction. Let
$T_{1},T_{2}:K\rightarrow E$ be two weakly inward and nonself uniformly
$L$-Lipschitzian total asymptotically nonexpansive mappings with squences
$\{\mu_{n}\},\{\lambda_{n}\}$ such that $%
{\textstyle\sum\nolimits_{n=1}^{\infty}}
\mu_{n}<\infty,%
{\textstyle\sum\nolimits_{n=1}^{\infty}}
\lambda_{n}<\infty$. Suppose that there exist $M,M^{\ast}>0,$ such that
$\phi(\kappa)\leq M^{\ast}\kappa$. Let $\{\alpha_{n}\}$ and $\left\{
\beta_{n}\right\}  $\ be two real sequences in $[\varepsilon,1-\varepsilon],$
for some $\varepsilon\in(0,1).$ Suppose that $\{x_{n}\}$ is genareted
iteratively by $(\ref{B})$. If one of $\ $the mappings $T_{1}$ and $T_{2}$ is
completely continuous and $F\neq\emptyset$, then the sequence $\left\{
x_{n}\right\}  $ converges strongly to a common fixed point of $T_{1}$ and
$T_{2}.$
\end{theorem}

\begin{proof}
From \ref{Lemma1}, $\lim_{n\rightarrow\infty}\left\Vert x_{n}-p\right\Vert $
exists for any $p\in F$. We must show that $\{x_{n}\}$ has a subsequence which
converges strongly to an element of fixed point set. From Lemma \ref{Lemma2},
we have $\lim_{n\rightarrow\infty}\left\Vert x_{n}-\left(  PT_{1}\right)
x_{n}\right\Vert =\lim_{n\rightarrow\infty}\left\Vert x_{n}-\left(
PT_{2}\right)  x_{n}\right\Vert =0$. Assume that $T_{1}$ is completely
continuous. By the nonexpansiveness of $P$, there exists subsequence
$\{PT_{1}x_{n_{j}}\}$ of $\{PT_{1}x_{n}\}$ such that $PT_{1}x_{n_{j}%
}\rightarrow p$. Thus $\left\Vert x_{n_{j}}-p\right\Vert \leq\left\Vert
x_{n_{j}}-PT_{1}x_{n_{j}}\right\Vert +$ $\left\Vert PT_{1}x_{n_{j}%
}-p\right\Vert $ gives $x_{n_{j}}\rightarrow p$ $(j\rightarrow\infty)$.
Similarly $\lim_{j\rightarrow\infty}\left\Vert x_{n_{j}}-\left(
PT_{1}\right)  x_{n_{j}}\right\Vert =0$ implis by continuity of $P$ and $T_{1}
$ that $p=PT_{1}p.$ The same way $p=PT_{2}p$. By Lemma \ref{F}, $p=T_{1}%
p=T_{2}p$. Since\ $F$ is closed, so $p\in F$. This completes the proof.
\end{proof}

\begin{theorem}
Let $K$ be a nonempty closed convex subset of a real uniformly convex and
smooth Banach space $E$ with $P$ as a sunny nonexpansive retraction. Let
$T_{1},T_{2}:K\rightarrow E$ be two weakly inward and nonself uniformly
$L$-Lipschitzian total asymptotically nonexpansive mappings with squences
$\{\mu_{n}\},\{\lambda_{n}\}$ such that $%
{\textstyle\sum\nolimits_{n=1}^{\infty}}
\mu_{n}<\infty,%
{\textstyle\sum\nolimits_{n=1}^{\infty}}
\lambda_{n}<\infty$. Suppose that there exist $M,M^{\ast}>0,$ such that
$\phi(\kappa)\leq M^{\ast}\kappa$. Let $\{\alpha_{n}\}$ and $\left\{
\beta_{n}\right\}  $\ be two real sequences in $[\varepsilon,1-\varepsilon],$
for some $\varepsilon\in(0,1).$ Suppose that $\{x_{n}\}$ is genareted
iteratively by $(\ref{B})$. If $T_{1}$ and $T_{2}$ satisfy condition
$(A^{^{\prime}})$ and $F\neq\emptyset$, then the sequence $\left\{
x_{n}\right\}  $ converges strongly to a common fixed point of $T_{1}$ and
$T_{2}.$
\end{theorem}

\begin{proof}
By Lemma \ref{Lemma1}, $\lim_{n\rightarrow\infty}d\left(  x_{n},F\right)  $
exists for all\textit{\ }$p\in F$. Also, by Lemma \ref{Lemma2},
\[
\lim_{n\rightarrow\infty}\left\Vert x_{n}-\left(  PT_{1}\right)
x_{n}\right\Vert =\lim_{n\rightarrow\infty}\left\Vert x_{n}-\left(
PT_{2}\right)  x_{n}\right\Vert =0.
\]
Using condition $(A^{\prime})$ and\ Lemma \ref{F}, we get%
\[
\lim_{n\rightarrow\infty}f\left(  d\left(  x_{n},F\right)  \right)  \leq
\lim_{n\rightarrow\infty}\left(  \frac{1}{2}\left(  \left\Vert x_{n}-\left(
PT_{1}\right)  x_{n}\right\Vert +\left\Vert x_{n}-\left(  PT_{2}\right)
x_{n}\right\Vert \right)  \right)  =0.
\]
Since $f$ is a nondecreasing function and $f(0)=0$, so we have $\lim
_{n\rightarrow\infty}d\left(  x,F\right)  =0.$ Now applying the theorem
\ref{Teorem}, we see $p\in F$.
\end{proof}

Our weak convergence theorem is as follows:

\begin{theorem}
Let $K$ be a nonempty closed convex subset of a real uniformly convex and
smooth Banach space $E$ satisfying Opial's condition with $P$ as a sunny
nonexpansive retraction. Let $T_{1},T_{2}:K\rightarrow E$ be two weakly inward
and nonself uniformly $L$-Lipschitzian total asymptotically nonexpansive
mappings with squences $\{\mu_{n}\},\{\lambda_{n}\}$ such that $%
{\textstyle\sum\nolimits_{n=1}^{\infty}}
\mu_{n}<\infty,%
{\textstyle\sum\nolimits_{n=1}^{\infty}}
\lambda_{n}<\infty$. Suppose that there exist $M,M^{\ast}>0,$ such that
$\phi(\kappa)\leq M^{\ast}\kappa$. Let $\{\alpha_{n}\}$ and $\left\{
\beta_{n}\right\}  $\ be two real sequences in $[\varepsilon,1-\varepsilon],$
for some $\varepsilon\in(0,1).$ Suppose that $\{x_{n}\}$ is genareted
iteratively by $(\ref{B})$. If $I-T_{2}$ and $I-T_{1}$ are demiclosed at zero
then the sequence $\left\{  x_{n}\right\}  $ converges weakly to a common
fixed point of $T_{1}$ and $T_{2}.$
\end{theorem}

\begin{proof}
Let $p\in F$ . Then by Lemma \ref{Lemma1}, $\lim_{n\rightarrow\infty
}\left\Vert x_{n}-p\right\Vert $ exists and $\left\{  x_{n}\right\}  $ is
bounded. It is point out that $PT_{1}$ and $PT_{2}$ are self-mappings defined
on $K$. We prove that $\left\{  x_{n}\right\}  $ converges subsequentially in
$F$. Uniformly convexity of Banach space $E$ implies that there exist two
weakly convergent subsequences $\left\{  x_{n_{i}}\right\}  $ and $\left\{
x_{n_{j}}\right\}  $ of baunded sequence $\left\{  x_{n}\right\}  $. Suppose
$w_{1}\in K$ and $w_{2}\in K$ are weak limits of the $\left\{  x_{n_{i}%
}\right\}  $ and $\left\{  x_{n_{j}}\right\}  $, respectively. Using Lemma
\ref{Lemma2}, we get $\lim_{n\rightarrow\infty}\left\Vert x_{n_{i}}-\left(
PT_{1}\right)  x_{n_{i}}\right\Vert =\lim_{n\rightarrow\infty}\left\Vert
x_{n_{i}}-\left(  PT_{2}\right)  x_{n_{i}}\right\Vert =0.$ Since $T_{1}$ is
demiclosed with respect to zero then we get $PT_{1}w_{1}=w_{1}$. Similarly,
$PT_{2}w_{1}=w_{1}$. That is, $w_{1}\in F$. In the same way, we have that
$w_{2}\in F$. Lemma \ref{F} guarantee that $p_{1,}~p_{2}\in F$.

Next, we give uniqueness. For this aim, assume that $w_{1}\neq w_{2}$. Using
Opial's condition, we have%
\begin{align*}
\lim_{n\rightarrow\infty}\left\Vert x_{n}-w_{1}\right\Vert  & =\lim
_{i\rightarrow\infty}\left\Vert x_{n_{i}}-w_{1}\right\Vert \\
& <\lim_{i\rightarrow\infty}\left\Vert x_{n_{i}}-w_{2}\right\Vert \\
& =\lim_{n\rightarrow\infty}\left\Vert x_{n}-w_{2}\right\Vert \\
& =\lim_{j\rightarrow\infty}\left\Vert x_{n_{j}}-w_{2}\right\Vert \\
& <\lim_{j\rightarrow\infty}\left\Vert x_{n_{j}}-w_{1}\right\Vert \\
& =\lim_{n\rightarrow\infty}\left\Vert x_{n}-w_{1}\right\Vert ,
\end{align*}
which is a contradiction. Thus $\left\{  x_{n}\right\}  $ converges weakly to
a point of $F$.
\end{proof}

Finally, we give an example which show that our theorems are applicable.

\begin{example}
Let $%
\mathbb{R}
$ be the real line with the usual norm $\left\vert {}\right\vert $ and let
$K=[-1,1]$. Define two mappings $T_{1},T_{2}:K\rightarrow K$ by
\[
T_{1}(x)=\left\{
\begin{array}
[c]{c}%
-2sin\frac{x}{2},\text{ }if~x\in\lbrack0,1]\\
2sin\frac{x}{2},\text{ }if\text{ }x\in\lbrack-1,0)
\end{array}
\right.  ~\text{\ and }T_{2}(x)=\left\{
\begin{array}
[c]{c}%
x,\text{ }if~x\in\lbrack0,1]\\
-x,\text{ }if\text{ }x\in\lbrack-1,0)
\end{array}
\right.  .
\]
In \cite{Guo}, the authors show that above mappings are asymptotically
nonexpansive mappings with common fixed point set $F=\left\{  0\right\}  $.
Since $T_{1}$~and $T_{1}$ are asymptotically nonexpansive mappings, then they
are uniformly L-Lipschitzian and total asymptotically nonexpansive mappings
with $F\neq\emptyset$. Consequently, our theorems are applicable.
\end{example}

\section{Conclusions}

1. We introduced a new concept of nonself total asymptotically nonexpansive
mapping which generalizes definition of some nonlinear mappings in existing literature.

2. Since the class of total asymptotically nonexpansive mappings includes
asymptotically nonexpansive mappings our results generalize results of Tukmen
et al. \cite{esref} and Akbulut et al. \cite{Akbulut}.

3. Our results also improve and extend the corresponding ones announced by
Ya.I. Alber et al.\cite{alber}, Chidume and Ofoedu \cite{chidume1, chidume2}
to a case of one mapping.

4. Our results generalize and extend results of Mukhamedov and Saburov \cite{FM1, FM2, FM3, FM4, FM5} and Khan et al. \cite{KKP} in view of
more general class of mappings.

\end{document}